\documentclass[12pt, a4paper, reqno]{amsart}
\usepackage{graphicx}
\usepackage{amssymb, amsthm, xcolor}
\usepackage[centering, heightrounded]{geometry}

\usepackage[small,bf]{caption}
\usepackage[subrefformat=parens]{subcaption}%labelformat=simple
\numberwithin{equation}{section}

\newtheorem{theorem}{Theorem}[section]
\newtheorem{lemma}[theorem]{Lemma}
\newtheorem{proposition}[theorem]{Proposition}

\theoremstyle{definition}
\newtheorem{definition}[theorem]{Definition}
\theoremstyle{remark}
\newtheorem{remark}[theorem]{Remark}

\newcommand{\R}{\mathbb{R}}

\newcommand{\Z}{\mathbb{Z}}
\newcommand{\Del}{\Delta}
\newcommand{\del}{\delta}
\newcommand{\om}{\omega}
\newcommand{\lam}{\lambda}
\newcommand{\Lam}{\Lambda}
\newcommand{\ve}{\varepsilon}
\newcommand{\pt}{\partial}
\newcommand{\ce}{\mathrel{\mathop:}=}

\newcommand{\la}{\langle}
\newcommand{\ra}{\rangle}
\newcommand{\rad}{\mathrm{rad}}
\newcommand{\Hr}{H_{\mathrm{rad}}}

\renewcommand{\l}{\mathopen{}\mathclose\bgroup\left}
\renewcommand{\r}{\aftergroup\egroup\right}

\def\norm[#1]{\left\Vert #1 \right\Vert}
\def\abs[#1]{\left\vert #1 \right\vert}

\def\tbra[#1,#2]{\left\langle #1 , #2\right\rangle} %

\title[Instability of stationary solutions in 1d]{Instability of stationary solutions for double power nonlinear Schr\"odinger equations in one dimension}

\author[N. Fukaya]{Noriyoshi Fukaya}
\address{%Department of Mathematics, Tokyo University of Science, Tokyo 162-8601, Japan
Waseda Research Institute for Science and Engineering, Waseda University, Tokyo 169-8555, Japan
\newline\indent
Osaka Central Advanced Mathematical Institute, Osaka Metropolitan University, 
Osaka 558-8585, Japan
\newline\indent
Institute for Mathematics and Computer Science, Tsuda University, Tokyo 187-8577, Japan}
\email{nfukaya@aoni.waseda.jp}

\author[M. Hayashi]{Masayuki Hayashi}
\address{Graduate School of Human and Environmental Studies,
Kyoto University, Kyoto 606-8501, Japan
\newline\indent
Waseda Research Institute for Science and Engineering, Waseda University, Tokyo 169-8555, Japan
}
 \email{hayashi.masayuki.3m@kyoto-u.ac.jp}


\begin{document}

\maketitle
\setcounter{tocdepth}{1}%{1}sectionまで表示

\begin{abstract}
We consider a double power nonlinear Schr\"odinger equation possessing the algebraically decaying stationary solution $\phi_0$ as well as exponentially decaying standing waves $e^{i\omega t}\phi_\omega(x)$ with $\omega>0$. According to the general theory, stability properties of standing waves are determined by the derivative of $\omega\mapsto M(\omega)\ce\frac{1}{2}\|\phi_\omega\|_{L^2}^2$; namely $e^{i\omega t}\phi_\omega$ with $\omega>0$ is stable if $M'(\omega)>0$ and unstable if $M'(\omega)<0$. However, the stability/instability of stationary solutions is outside the general theory from the viewpoint of spectral properties of linearized operators. In this paper we prove the instability of the stationary solution $\phi_0$ in one dimension under the condition $\lim_{\omega\downarrow 0}M'(\omega)\in[-\infty, 0)$. The key in the proof is the construction of the one-sided derivative of $\omega\mapsto\phi_\omega$ at $\omega=0$, which is effectively used to construct the unstable direction.
\end{abstract}
%which is a canonical model possessing the algebraically decaying stationary solution $\phi_0$ as well as exponentially decaying standing waves $e^{i\omega t}\phi_\omega(x)$ with $\om>0$. 
%a double power nonlinear Schr\"odinger equation in one dimension and $L^2$-subcritical setting
%the standing wave with positive frequency is stable

\tableofcontents

\section{Introduction}

We consider the following nonlinear Schr\"odinger equation with double power nonlinearities:
\begin{equation}\label{eq:1.1}
    i\pt_t u
    =-\Del u
    +|u|^{p-1}u
    -|u|^{q-1}u,\quad 
    (t,x)\in\R\times\R^d, 
\end{equation}
where $1<p<q<1+4/(d-2)_+$ and $u$ is the complex-valued unknown function of $(t,x)$. It is well-known that the Cauchy problem of \eqref{eq:1.1} is locally well-posed in the energy space $H^1(\R^d)\ce H^1(\R^d; \mathbb{C})$ (see \cite[Chapter 4]{C03}). Moreover, the energy
\begin{align*}
    E(u(t))
    \ce\frac{1}{2}\|\nabla u(t)\|_{L^2}^2
    +\frac{1}{p+1}\|u(t)\|_{L^{p+1}}^{p+1}
    -\frac{1}{q+1}\|u(t)\|_{L^{q+1}}^{q+1}
\end{align*}
and the charge $\frac{1}{2}\|u(t)\|_{L^2}^2$ are conserved by the flow of \eqref{eq:1.1}. If we consider the case $1<p<q<1+4/d$, which is of our interest, then it follows from the conservation laws and the suitable Gagliardo--Nirenberg inequality that the solution for any initial data in $H^1(\R^d)$ exists globally in time (see \cite[Chapter 6]{C03}). 
%the global well-posedness in $H^1(\R^d)$ holds for the Cauchy problem of \eqref{eq:1.1}]
%In particular, no blowup occurs in this case.

Equation \eqref{eq:1.1} has standing wave solutions of the form
\[  u_\om(t, x)
    =e^{i\om t}\phi_\om(x), \]
where $\om\ge 0$ and $\phi_\om$ is the unique positive, radial, and decreasing solution (ground state) of the stationary equation
\begin{equation}\label{eq:1.2}
  -\Del \phi 
  +\om\phi 
  +|\phi|^{p-1}\phi
  -|\phi|^{q-1}\phi
  =0,\quad 
  x\in\R^d. 
\end{equation}
We refer to \cite{BL83-1} for existence, \cite{LN93} for radial symmetry properties, and \cite{PS98, ST00} for uniqueness of the ground states. 
We note that \eqref{eq:1.2} is rewritten by using the action functional as $S_\om'(\phi)=0$, where $S_\om(\phi)$ is defined by
\begin{align*}
S_\om(\phi)&=E(\phi)+\frac{\om}{2}\norm[\phi]_{L^2}^2
\\
&=\frac{1}{2}\|\nabla\phi\|_{L^2}^2+\frac{\om}{2}\norm[\phi]_{L^2}^2+\frac{1}{p+1}\|\phi\|_{L^{p+1}}^{p+1}-\frac{1}{q+1}\|\phi\|_{L^{q+1}}^{q+1}.
\end{align*}
From the classical variational argument, the profile $\phi_\om$ is found in the space $H^1(\R^d)$ if $\om>0$ and in $(\dot{H}^1\cap L^{p+1})(\R^d)$ if $\om=0$ (see \cite{FH21}). Moreover, it is known that $\phi_\om$ with $\om>0$ decays exponentially at infinity (see, e.g., \cite{BL83-1}). On the other hand, the stationary solution $\phi_0$ decays algebraically:
\begin{equation} \label{eq:1.3}
    \phi_0(r)
    \sim\left\{
    \begin{aligned}
   &r^{-2/(p-1)}&
   &\text{if }1<p<d/(d-2)_+,
\\ &r^{-(d-2)}(\log r)^{-(d-2)/2}&
   &\text{if }p=d/(d-2)_+,
\\ &r^{-(d-2)}&
   &\text{if }p>d/(d-2)_+
    \end{aligned}
    \right.
\end{equation}
as $r\to\infty$ (see \cite{V81}), where we write
\[  f(r)\sim g(r) \text{ as }r\to l \]
for $l\in[-\infty,\infty]$ if there exists a positive constant $c_0$ such that 
\[  \lim_{r\to l}\frac{f(r)}{g(r)}
    =c_0.  \]
The sharp decay estimate \eqref{eq:1.3} implies that $\phi_0\in L^2(\R^d)$ if and only if 
\begin{equation}
\label{eq:1.4}
  \begin{aligned}
   &1<p<1+4/d&
   &\text{if }1\le d\le 3,
\\ &1<p\le 1+4/d&
   &\text{if }d=4,
\\ &1<p< 1+4/(d-2)_+&
   &\text{if }d\ge 5,
  \end{aligned}
\end{equation}
see \cite[Corollary 1.4]{FH21} for more details.

%In this paper we consider the stability of standing wave solutions $e^{i\om t}\phi_\om$ of \eqref{eq:1.1}, in particular, the stationary solution $\phi_0$. The definition of the stability is given as follows.

We now give the definition of stability/instability of standing waves for \eqref{eq:1.1}.
\begin{definition}\label{def:1.1}
We assume \eqref{eq:1.4} if $\omega=0$. The standing wave $e^{i\om t}\phi_\om$ of \eqref{eq:1.1} is \emph{stable} if for any $\ve>0$ there exists $\delta>0$ such that if $u_0\in H^1(\R^d)$ satisfies $\|u_0-\phi_\om\|_{H^1}<\delta$, then the $H^1$-solution $u(t)$ of \eqref{eq:1.1} with $u(0)=u_0$ satisfies 
\begin{align*}
    \inf_{(\theta, y)\in\R\times\R^d}\|u(t)-e^{i\theta}\phi_\om(\cdot-y)\|_{H^1}<\ve
\end{align*}
for all $t\in\R$. Otherwise, the standing wave is \emph{unstable}.
\end{definition}
%%%%%%%
%We set $M(\om)\ce\frac{1}{2}\|\phi_\om\|_{L^2}^2$
%\[ 
%M(\om)
%    \ce\frac{1}{2}\|\phi_\om\|_{L^2}^2\quad 
%    \text{for $\om>0$}. 
%\] 
It is known from the abstract theory of \cite{GSS87} (see also \cite{SS85, W85, W86}) that stability properties of standing waves are determined by the derivative of
\begin{align*}
\omega\mapsto M(\omega)\ce\frac{1}{2}\|\phi_\om\|_{L^2}^2
\quad \text{for $\om>0$},
\end{align*}
namely the standing wave $e^{i\om t}\phi_\om$ with $\om>0$ is stable if $M'(\om)>0$ and unstable if $M'(\om)<0$. 
However, it is delicate in general to determine the sign of $M'(\om)$ for double power nonlinearities because of the lack of scaling symmetries. 
%which gives a major difference with the case of pure power nonlinearities at the technical level. 

The case $1<p<q<1+4/d$ is interesting to consider the stability/instability problem in \eqref{eq:1.1} because stability properties of standing waves may change with frequency $\omega$ even fixed exponents $p$ and $q$ in this case. By perturbation arguments from pure power nonlinearities one can prove that the standing wave $e^{i\om t}\phi_\om$ is stable for large $\om>0$ (see \cite{F03}). However, $e^{i\om t}\phi_\om$ with small $\om>0$ can be unstable, and indeed it was proved in \cite{O95dp} that if $d=1$ and $p+q>6$, then $e^{i\om t}\phi_\om$ is unstable for small $\om>0$. The authors \cite{FH21} improved this sufficient condition and proved that if
\begin{align}
\label{eq:1.5}
    q>\gamma_d(p)
    \ce\frac{16+d^2+6d-pd(d+2)}{d(d+2-(d-2)p)}
\end{align} 
%including the case $\om=0$,
%%and $\om\ge 0$ is sufficiently small, 
then $e^{i\om t}\phi_\om$ is unstable for small  $\om\ge0$ 
(see also Figure \ref{fig:1}\subref{fig:1a}).
We note that the condition \eqref{eq:1.5} is characterized in terms of stationary solutions as
\begin{align*}
\eqref{eq:1.5}\iff \l.\pt_\lam^2 S_0(\phi_0^\lam)\r|_{\lam=1}<0\quad
%\text{for}~
(\phi_0^\lam(x)\ce\lambda^{d/2}\phi_0(\lam x) ).
\end{align*}
%Although there exist no blow up solutions in \eqref{eq:1.1} with $q<1+4/d$, 
The results in \cite{O95dp, FH21} show the existence of unstable standing waves in the absence of blowup solutions, which gives a major difference with the case of pure power nonlinearities\footnote{It is well-known for the pure power nonlinear Schr\"odinger equations that the standing wave $e^{i\om t}\phi_\om$ for all $\om>0$ is stable if the exponent $p$ of the nonlinearity satisfies $p<1+4/d$ (see \cite{CL82}), and unstable by blowup if $1+4/d\le p<1+4/(d-2)_+$ (see \cite{BC81, W82}). }
 and implies that \eqref{eq:1.1} has a richer dynamics.
%
%This property gives a major difference with the case of pure power nonlinearities because 

%この辺りに条件の図を追加．図は前に作成したものをつかえばよい．モノクロに対応させたほうがよいか．
\begin{figure}[ht]
\begin{minipage}[b]{0.45\linewidth}
\includegraphics[width=\linewidth]{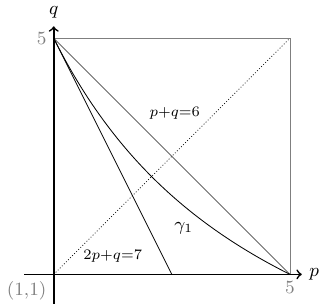}
\subcaption{Comparison of three conditions}
\label{fig:1a}%%
\end{minipage}
\begin{minipage}[b]{0.45\linewidth}
\includegraphics[width=\linewidth]{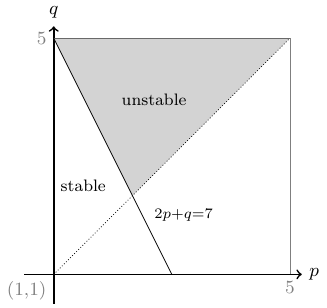}
\subcaption{Stability/instability with small $\om>0$}
\label{fig:1b}%%
\end{minipage}
\caption{One can compare the conditions of \cite{O95dp, FH21} with the sharp condition in the left picture. We note that the right picture does not describe the stability/instability with $\om=0$.}
\label{fig:1}%%
\end{figure}
Recently, it was independently proved in \cite{H21pre, KLT22} that when $d=1$, there exists the limit of $M'(\om)$ as $\om\downarrow0$ and\begin{equation} \label{eq:1.6}
    M'(0)
    \ce\lim_{\om\downarrow 0}M'(\om)
    \left\{
    \begin{aligned}
    &>0&
    &\text{if }2p+q<7,
\\  &=0&
    &\text{if }2p+q=7,
\\  &<0&
    &\text{if }2p+q>7 \text{ and }p<7/3,
\\  &=-\infty&
    &\text{if }p\ge 7/3,
    \end{aligned}\right.
\end{equation}
%%%%%%%%%%%
%%%%%%%%%%%
which in particular implies that the condition $2p+q>7$ is sharp for yielding instability of standing waves with small positive frequencies 
(see Figure \ref{fig:1}\subref{fig:1b}). 
We note that in \cite{O95dp, H21pre, KLT22} the proof of stability/instability is done by applying the abstract theory \cite{GSS87}, which is not applicable for the case $\om=0$ due to the lack of coercivity properties of linearized operators.
% (see, e.g., \cite{FH22}). 
On the other hand, in \cite{FH21} the instability is proved by combining variational characterization, virial identities, and the condition $\l.\pt_\lam^2 S_\om(\phi_\om^\lam)\r|_{\lam=1}<0$, which is also applicable to the case $\om=0$. 
While the analysis based on the condition $\l.\pt_\lam^2 S_\om(\phi_\om^\lam)\r|_{\lam=1}<0$ is applicable to many types of equations (see \cite{O95ds, FO03, CO14, FO18, FO19}), 
%has a wide range of applications
this condition is not expected in general to be sharp for yielding the instability. Indeed, in our case the condition of \cite{FH21} with $d=1$
\begin{align}
\label{eq:1.7}
q>\gamma_1(p)=\frac{23-3p}{3+p}
\iff \l.\pt_\lam^2 S_0(\phi_0^\lam)\r|_{\lam=1}<0
\end{align}
has a gap with the condition
\begin{align}
\label{eq:1.8}
2p+q>7\iff M'(0)\in[-\infty,0).
\end{align}
By analogy with the stability/instability theory with positive frequencies, it would be natural to conjecture that the statement
\begin{equation}
 \label{eq:1.9}
    M'(0)\in[-\infty, 0)\implies\text{$\phi_0$ is unstable}
%    \quad\text{or equivalently }2p+q>7. 
\end{equation}
% which is expected to
holds in a general setting and that \eqref{eq:1.9} is sharp for the instability of stationary solutions. In this paper we prove that \eqref{eq:1.9} is true for \eqref{eq:1.1} with $d=1$ and fill the remaining gap between \eqref{eq:1.7} and \eqref{eq:1.8}. 
%%%%%%%%%

%However, the instability of $\phi_0$ under \eqref{eq:1.9} was an open problem because one cannot treat the case $\om=0$ by the abstract framework of \cite{GSS87}. The most difficult part is to make meaning of the derivative $\frac{d\phi_\om}{d\om}|_{\om=0}$, which plays an important role in construction of an unstable direction. The aim of this paper is to prove the instability of $\phi_0$ under $d=1$ and \eqref{eq:1.9}.
%%%%%%%%%%%%%%%%%%
%%%%%%%%%%%%%%%%%%

We now state our main results of this paper.
We set the function
\begin{align} 
\label{eq:1.10}
  \eta_\om
  &\ce \pt_\om\phi_\om\quad 
  \text{for $\om>0$}.
\end{align}
It follows from \eqref{eq:1.2} that $\eta_\om$
satisfies the equation
\begin{align}
\label{eq:1.11}
L_\om\eta_\om =-\phi_\om\quad\text{in}~\R^d\quad(\om>0),
\end{align}
where $L_\om$ is the linearized operator around $\phi_\om$ defined by
\begin{align*}
L_\om=S_\om''(\phi_\om)|_{H^1(\R^d; \R)}= -\Delta +\om+p\phi_\om^{p-1}-q\phi_\om^{q-1}\quad\text{for}~\om\ge0.
\end{align*}
We consider the limit problems of \eqref{eq:1.10} and \eqref{eq:1.11} as $\om\downarrow0$ in the case of $d=1$. Our first result is the following.
%%%%%%%
%%%%%%%
\begin{theorem}\label{thm:1.2}
Let $d=1$ and $1<p<q$. Then there exists the smooth even function $\eta_0$ on $\R$ such that $\eta_\omega(x)\to\eta_0(x)$ as $\omega\downarrow0$ for all $x\in\R$. Moreover, $\eta_0$ satisfies the equation
\begin{align}
\label{eq:1.12}
L_0\eta_0=-\phi_0\quad\text{\rm in}~\R,
\end{align}
and   
\begin{align}
    \label{eq:1.13}
    \eta_0(x)
   &\sim -x^{-2/(p-1) +2},
\\ \label{eq:1.14} 
    \eta_0'(x)
   &=O(x^{-2/(p-1) +1})
\end{align}
as $x\to\infty$. 
%In partucular, the function $\phi_0\eta_0$ is integrable if $p<7/3$ and nonintegrable if $p\ge7/3$ on $\R$. 
\end{theorem}
%%%%%%%%%%%%%%%
When $d=1$, $\phi_\om$ can be explicitly represented in terms of a certain inverse function by quadrature. The explicit formula of $\phi_\om$ with $\om>0$ is already used in \cite{IK93}, but we newly derive the limit function of $\eta_\om$ as $\om\downarrow0$ here.
% $\pt_\om\phi_\om$ as
It is easily verified that the function $\eta_0$ satisfies 
%for each $x\in\R$
\begin{align}
\label{eq:1.15}
\eta_0(x)=\lim_{\om\downarrow0}\frac{\phi_\om(x)-\phi_0(x)}{\omega}
\quad\text{for}~x\in\R,
\end{align}
which turns out that $\eta_0$ gives the one-sided derivative of $\om\mapsto\phi_\om$ at $\om=0$. We recall that $\phi_\om$ with $\om>0$ decays exponentially and $\phi_0$ decays algebraically, so one can see from \eqref{eq:1.15} that the derivation of $\eta_0$ is highly nontrivial even for the one-dimensional case. 
%so it is highly nontrivial to derive the limit function of the RHS on \eqref{eq:1.15} 
%Indeed, indefinite forms appear in several steps
When $q=2p-1$, the function $\eta_0$ is completely expressed in terms of elementary functions (see Appendix \ref{sec:B}), which is of independent interest.

We use the function $\eta_0$ to construct the unstable direction of $\phi_0$. Here we quickly review the instability theory for $e^{i\om t}\phi_\om$ with $\om>0$. The proof of the instability is reduced to find the function $\psi\in H^1_{\rm rad}(\R^d)$ satisfying
\begin{align}
\label{eq:1.16}
%\exists\psi\in H^1_{\rm rad}(\R^d)\quad\text{s.t.}
%\quad 
(\psi, \phi_\om)_{L^2}
  =0\quad\text{and}\quad\tbra[L_\om\psi, \psi] <0.
\end{align}
%When we consider the case $\om>0$, we set 
We consider the function
\begin{align*}
\psi_\omega=\phi_\omega-\frac{\norm[\phi_\omega]_{L^2}^2}{M'(\omega)}\eta_\omega
\quad\text{for}~\omega>0,
\end{align*}
which belongs to $H^1_{\rm rad}(\R^d)$ and satisfies $(\psi_\om, \phi_\om)_{L^2}=0$. Moreover, a direct calculation shows that
\begin{align}
\label{eq:1.17}
\tbra[L_\omega\psi_\omega,\psi_\omega]
=\tbra[L_\omega\phi_\omega,\phi_\omega]+\frac{\norm[\phi_\omega]_{L^2}^4}{M'(\omega)}.
\end{align}
Therefore, if we assume $M'(\om)<0$, then it follows from \eqref{eq:1.17} and the fact $\tbra[L_\om\phi_\om,\phi_\om]<0$ that $\psi_\om$ satisfies the condition \eqref{eq:1.16}.

From this observation, the function 
\begin{align*}
\psi_0=\phi_0-\frac{\norm[\phi_0]_{L^2}^2}{M'(0)}\eta_0
\quad \bigl(\eta_0=\lim_{\omega\downarrow0}\eta_\omega\bigr)
\end{align*}
is a natural candidate satisfying \eqref{eq:1.16} with $\om=0$. However, one can see from \eqref{eq:1.13} that $\eta_0$ does not belong to $H^1_{\rm rad}(\R)$ in general (or can even grow at infinity) even if we consider the case of $\phi_0\in H^1_{\rm rad}(\R)$.\footnote{This situation is very different from the case of $\om>0$ because $\eta_\om$ as well as $\phi_\om$ decays exponentially when $\om>0$.} 
To solve this problem, we use cut-off functions to approximate $\eta_0$ by the localized function, and then obtain the following result. 
%In particular we obtain the following result.  
%by using cut-off functions modify the function $\psi_0$ 
%%%%%%
\begin{theorem}\label{thm:1.3}
Let $d=1$ and $1<p<q<5$. If $M'(0)\in[-\infty,0)$, then there exists $\psi\in H_{\rad}^1(\R)$ satisfying \eqref{eq:1.16} with $\om=0$.
\end{theorem}
%\vspace{10pt}
The key observation is that the RHS on \eqref{eq:1.17} is still defined for $\om=0$ and $M'(0)\neq0$. Actually we prove that \eqref{eq:1.17} is true for $\om=0$ and $M'(0)\neq0$ in a certain limiting sense (see Theorem \ref{thm:4.1}). 
%under the assumption of Theorem \ref{thm:1.3}
In the proof of this, we essentially use the sharp decay/growth estimate of $\eta_0$ to handle error terms coming from the cut-off function.  

As a consequence of Theorem \ref{thm:1.3}, we obtain the following instability result. 
\begin{theorem} \label{cor:1.4}
Let $d=1$ and $1<p<q<5$. If $2p+q>7$, then the stationary solution $\phi_0$ of \eqref{eq:1.1} is unstable.
\end{theorem}
%%%%%%
%For the proof of the instability we use varitional characterization of $\phi_0$, which plays an alternative role of coercivity of linearized operators. 
%%the existence of $\psi\in H^1_{\rm rad}(\R^d)$ satisfying \eqref{eq:1.16} implies the instability
%%%%%%%%%%%%%%%%

Although we cannot use the argument based on coercivity properties of linearized operators as in \cite{W85, W86, GSS87}, we instead apply the variational approach in \cite{SS85, O14}, which still works for stationary solutions (the case $\om=0$).
For the convenience of the reader, in Section \ref{sec:A} we give a proof of the fact that the existence of $\psi\in H^1_{\rm rad}(\R^d)$ satisfying \eqref{eq:1.16} implies the instability in our setting, which in particular contains the proof of Theorem \ref{cor:1.4}.
% including the case $\om=0$.
%%%%
%This follows from Theorem~\ref{thm:1.3} and the argument of \cite{O14}, which treats a two-parameter family of solitary waves. 

As relevant works \cite{NOW17, LN18}, the authors study instability of algebraic decaying traveling waves for nonlinear Schr\"odinger equations with derivative coupling, which correspond to stationary solutions of \eqref{eq:1.1} from the viewpoint of elliptic equations of the profiles. However, they use scaling symmetries of equations and the specific information of a two-parameter family of solitary waves, which cannot be applied in our setting. Our approach in this paper has a natural correspondence with the general instability theory \cite{GSS87} and is more robust.
%the general instability theory of standing waves with positive frequencies

The rest of the paper is organized as follows. In Section \ref{sec:2} we first derive the basic properties of $\phi_\om$ and construct the key function $\eta_0$. Then we derive the sharp decay/growth estimate of $\eta_0$ and complete the proof of Theorem \ref{thm:1.2}. In Section \ref{sec:3} we establish the link between $M'(0)$ and $\eta_0$, which plays an important role to bridge between the previous works \cite{H21pre, KLT22} and our analysis.
%connect Theorem \ref{thm:1.2} and Theorem \ref{thm:1.3}. 
In Section \ref{sec:4} we construct the unstable direction in the sense of \eqref{eq:1.16} based on the function $\eta_0$, and prove Theorem \ref{thm:1.3}.

\section{Construction of the key function $\eta_0$}
%Sharp decay/growth estimates of $\eta_0$
\label{sec:2}

In this section we construct the key function $\eta_0$ and derive the sharp decay/growth estimates of $\eta_0$.
%we prove Theorem~\ref{thm:1.2}. 
Throughout this section we assume $d=1$ and $1<p<q$. 

%by quadrature
We first give the explicit formula of $\phi_\om$ in terms of a certain inverse function (see \eqref{eq:2.5}), and then derive the limit function of $\pt_\om\phi_\om$ as $\omega\downarrow0$. We recall that $\phi_\omega$ is the unique positive, even, and decreasing solution of the equation
\begin{align}
\label{eq:2.1}
-\phi''+\om\phi +|\phi|^{p-1}\phi-|\phi|^{q-1}\phi=0,\quad x\in\R.
%,~1<p<q.
\end{align}
%We use 
We define the functions $f$ and $W$ by
\begin{align*}
\begin{aligned}
    f(s)
   &=s^{(p-1)/2}
    -s^{(q-1)/2},
\\  W(s; \om)
   &=\om s
    +\frac{2}{p+1}s^{(p+1)/2}
    -\frac{2}{q+1}s^{(q+1)/2}
\end{aligned}
\quad\text{for $s\ge 0$ and $\om\in \R$}.
\end{align*}
We note that $W_s(s; \om)=\om+f(s)$. From the shape of the function $s\mapsto W(s; \om)$, we see that for each $\om\ge0$ there exists the unique positive zero point $a(\om)>0$ of $W(\cdot; \om)$, and $W(\cdot; \om)$ satisfies
\begin{align}\label{eq:2.2}
    W(s; \om)
    &>0\quad 
    \text{on }(0, a(\om)),&
    W(a(\om); \om)
    &=0,&
    W_s(a(\om); \om)
    &<0.&
\end{align}
By the implicit function theorem and $W\in C^\infty((0, \infty)\times\R)$, we deduce that $a\in C^\infty([0, \infty))$. 
%%%%%
\begin{lemma}
\label{lem:2.1}
For $\om\ge0$ and $x\ge0$, we have the relation
%$\phi_\om(x)$ is represented as
\begin{equation} \label{eq:2.3}
  \phi_\om'(x)
  =-\sqrt{W(\phi_\om(x)^2; \om)},
%  \quad x\ge0
\end{equation}
and $a(\om)=\phi_\om(0)^2$.
% for $\omega\ge0$. 
\end{lemma}

\begin{proof}
Multiplying $2\phi_\om'$ with the equation \eqref{eq:2.1}, we have
\[  2\phi_\om'\phi_\om''
    =2\om \phi_\om\phi_\om'
    +2f(\phi_\om^2)\phi_\om\phi_\om',\quad x\in\R.  \]
This is rewritten as
\begin{align*}
    (\phi_\om'^2)'
    =\om(\phi_\om^2)'
    +f(\phi_\om^2)(\phi_\om^2)'
    =W(\phi_\om(x)^2; \om)'.
\end{align*}
Integrating both sides on the interval $(x, \infty)$ and using the fact that $\phi_\om(x)$, $\phi_\om'(x)\to 0$ as $x\to\infty$, we obtain 
\begin{align*}
    \phi_\om'(x)^2
    =W(\phi_\om(x)^2; \om),\quad 
    x\ge0. 
\end{align*}
This implies \eqref{eq:2.3} because $\phi_\om'\le 0$ on $[0,\infty)$.

The evenness of $\phi_\om$ implies $\phi_\om'(0)=0$. Thus we have $W(\phi_\om(0)^2; \om)=0$. Since $s\mapsto W(s; \om)$ has the unique positive zero point $a(\omega)$, we deduce $a(\om)=\phi_\om(0)^2$. 
\end{proof}
%%%%%
We set 
\[  F(\tau; \om)
    \ce \int_\tau^1\frac{ds}{\sqrt{s W(a(\om)s; \om)}}\quad 
    \text{for }\tau\in(0,1]~\text{and}~\om\ge 0. 
    \]
We note from \eqref{eq:2.2} that $W(a(\om)s; \om)\sim 1-s$ as $s\uparrow 1$. This yields that the integral has a finite value for $\tau\in(0, 1)$ and $\tau\mapsto F(\tau; \om)$ is continuous on $(0,1]$. From the partial derivative
% with respect to $\tau$
\begin{align}\label{eq:2.4}
  F_\tau(\tau; \omega)
  =-\frac{1}{\sqrt{\tau W(a(\omega)\tau; \omega)}}
  <0\quad 
  \text{for}~\tau\in(0,1),
\end{align}
we deduce that $\tau\mapsto F(\tau; \om)$ is strictly decreasing. Moreover, it follows from the asymptotic behavior of $s W(s; \om)$ at $s=0$ that
\begin{align*}
F(\tau; \om)\underset{\tau\downarrow0}{\sim}
\l\{
\begin{aligned}
&{-}\om\log \tau&&\text{if}~\omega>0,
\\
&\,\tau^{-(p-1)/4}&&\text{if}~\omega=0.
\end{aligned}
\r.
%\text{as}~\tau\downarrow 0.
\end{align*}
Therefore, for each $\om\in[0,\infty)$ the function $F(\cdot;\om)$ is bijective from $(0, 1]$ to $[0, \infty)$ and its inverse function $F^{-1}(\cdot; \om)\colon [0, \infty)\to(0, 1]$ is defined. We set
\begin{align*}
G(z; \om)\ce F^{-1}(z; \om)
\quad\text{for}~z\ge0~\text{and}~\omega\ge0. 
\end{align*}
%%%%%%%%%
%%%%%%%%%
\begin{lemma}
\label{lem:2.2}
For $\om\ge0$ and $x\ge0$, we have
\begin{equation}\label{eq:2.5}  
 \phi_\om(x)^2
  =a(\om)G(b(\om)x; \om),
%  \quad x\ge 0, 
\end{equation}
where $b(\om)\ce\frac{2}{\sqrt{a(\om)}}$. In particular, $\phi_\om(x)\to\phi_0(x)$ as $\omega\downarrow0$ for all $x\in\R$.
\end{lemma}

\begin{proof}
By dividing \eqref{eq:2.3} on both sides by $-\sqrt{W(\phi_\om(x)^2; \om)}$, integrating on $[0,x]$, and the change of the variables $\phi_\om(y)^2=a(\om)s$, we obtain
\begin{align*}
  x&=\int_{0}^{x}\frac{-\phi_\om'(y)}{\sqrt{W(\phi_\om^2(y);\om)}}dy
  =\frac{\sqrt{a(\om)}}{2}\int_{\phi_\om^2(x)/a(\om)}^{\phi_\om^2(0)/a(\om)}\frac{ds}{\sqrt{s W(a(\om)s; \om)}}.
\end{align*}
Since $\phi_\om(0)^2=a(\om)$, we have 
\[b(\om) x
  =F\left(\frac{\phi_\om(x)^2}{a(\om)}; \om\right), \]
which yields \eqref{eq:2.5}. The latter claim easily follows from the continuities of $a(\omega)$ and $G(z;\omega)$.
%Since $G(\cdot; \om)$ is the inverse function of $F(\cdot; \om)$, we obtain the conclusion.
\end{proof}
The formula \eqref{eq:2.5} gives a simple proof of the sharp decay of $\phi_0(x)$ as $x\to\infty$.
\begin{lemma}
\label{lem:2.3}
$\phi_0(x)\sim x^{-2/(p-1)}$ as $x\to\infty$. 
%\underset{x\to\infty}{\sim} 
\end{lemma}
\begin{proof}
Since $F(\tau; 0)\sim \tau^{-(p-1)/4}$ as $\tau\downarrow 0$, we have
$
G(z; 0){\sim} z^{-4/(p-1)}
$
as $z\to\infty$. From \eqref{eq:2.5}, we obtain 
\begin{align*}
\phi_0(x)\sim G(b(0)x; 0)^{1/2}\sim x^{-2/(p-1)}
\end{align*}
as $x\to\infty$. 
\end{proof}
%%%%%%%%%%
%%%%%%%%%%
Differentiating \eqref{eq:2.5} with respect to $\om$, we have 
\begin{equation}\label{eq:2.6}
  \begin{aligned}
   &2\eta_\omega(x)\phi_\omega(x)
\\ &\quad 
    =a'(\om)G(b(\om)x; \om)
    +a(\om)b'(\om)xG_z(b(\om)x; \om)
    +a(\om)G_\om(b(\om)x; \om).
  \end{aligned}
\end{equation}
We now check the definedness of each term in the RHS of \eqref{eq:2.6} with $\om=0$. Since $a\in C^\infty([0, \infty))$ and $a(\om)>0$, $a'(0)$ and $b'(0)$ are defined. By \eqref{eq:2.4}, we see that $z\mapsto G(z; \om)$ is differentiable for $z>0$ and $\om\ge 0$ and 
\begin{align}\label{eq:2.7}
    G_z(z; \om)
   &=\frac{1}{F_\tau(G(z; \om); \om)}
    =-\sqrt{G(z; \om)W(a(\om)G(z; \om); \om)},
\end{align}
whose rightmost side is still defined with $z=0$. The formula
\begin{equation*}
%%label{eq:2.8}%
  a'(\om)
  =-\frac{W_\om(a(\om); \om)}{W_s(a(\om); \om)}
  =-\frac{a(\om)}{W_s(a(\om); \om)}
\end{equation*} 
%where we used $W_\om(s; \om)=s$, we have
and a direct calculation yield that 
\begin{align*}
    \pt_\om \frac{1}{\sqrt{s W(a(\omega)s; \omega)}}
   &=\frac{\{a'(\om)W_s(a(\om)s; \om)+a(\om)\}\sqrt{s}}{2W(a(\om)s; \om)^{3/2}} 
\\ &=\frac{a(\om)\{-W_s(a(\om)s; \om)+W_s(a(\om); \om)\}\sqrt{s}}{2W_s(a(\om); \om)W(a(\om)s; \om)^{3/2}}.
\end{align*}

We note that for any $\tau\in(0, 1]$ there exist $\delta>0$ and $C>0$ such that 
\begin{equation} \label{eq:n2.8}
    \left|\frac{a(\om)\{-W_s(a(\om)s; \om)+W_s(a(\om); \om)\}\sqrt{s}}{2W_s(a(\om); \om)W(a(\om)s; \om)^{3/2}}\right|
    \le \frac{C}{\sqrt{1-s}}
\end{equation}
for all $\omega\in[0, \delta]$ and $s\in[\tau, 1)$. Indeed, by $W_s(a(0); 0)<0$ and the continuity of $(s, \omega)\mapsto W_s(s; \omega)$, there exists $\delta>0$ such that $-W_s(a(\omega)\theta; \omega)\ge\delta$ for all $\theta\in[1-\delta, 1]$ and $\omega\in[0, \delta]$. This implies that 
\begin{align*}
    W(a(\omega)s; \omega)
    =-a(\omega)\int_s^1W_s(a(\omega)\theta; \omega)\,d\theta
    \gtrsim 1-s
\end{align*}
for $s\in[1-\delta, 1]$ and $\omega\in[0, \delta]$. Moreover, by the boundedness of $(\theta, \omega)\mapsto W_{ss}(a(\omega)\theta; \omega)$ on $[1-\delta, 1]\times[0, \delta]$, we obtain
\[  \lvert-W_s(a(\om)s; \om)+W_s(a(\om); \om)\rvert
    \lesssim 1-s  \]
for $s\in[1-\delta, 1]$ and $\omega\in[0, \delta]$. Therefore, \eqref{eq:n2.8} holds for all $s\in[1-\delta, 1)$ and $\omega\in[0, \delta]$. 
One can easily see that \eqref{eq:n2.8} holds for $s\in[\tau, 1-\delta]$ and $\omega\in[0, \delta]$ because $(s, \omega)\mapsto W(a(\omega)s; \omega)$ has the positive minimum on $[\tau, 1-\delta]\times[0, \delta]$. 

By \eqref{eq:n2.8} and the dominated convergence theorem we deduce that the partial derivative $F_\om(\tau; \om)$ exists for $\tau\in(0, 1]$ and $\om\ge0$, and that
\begin{align*}
  F_\om(\tau; \om)
  =-\frac{a(\om)}{2W_s(a(\om); \om)}\int_\tau^1\frac{\{-W_s(a(\om)s; \om)+W_s(a(\om); \om)\}\sqrt{s}}{W(a(\om)s; \om)^{3/2}}\,ds.
\end{align*}
Moreover, by differentiating $F(G(z; \omega); \omega)=z$ with respect to $\omega$, we have
\begin{align}
\label{eq:2.8}
G_\omega(z; \omega) 
  =-\frac{F_\omega(G(z; \omega) ;\omega)}{F_\tau(G(z; \omega); \omega)}.
\end{align}
From this we see that the partial derivative $G_\om(z; \om)$ is defined for $z\ge0$ and $\om\ge0$.

Now, by passing the limit $\omega\downarrow0$ in \eqref{eq:2.6}, we can define the function $\eta_0$ by
\begin{align}
\label{eq:2.9}
    \eta_0(x)=\lim_{\omega\downarrow0}\eta_\omega(x)
    =\frac{1}{2\phi_0(x)}\{
  \begin{aligned}[t]
   &a'(0)G(b(0)x; 0)
\\ &+a(0)b'(0)xG_z(b(0)x; 0)
    +a(0)G_\om(b(0)x; 0)\}
  \end{aligned}
\end{align}
for $x\in\R$. It follows from this formula that $\eta_0\in C(\R)$.
%%%%%%%%%%%
%%%%%%%%%%%

%The following lemma is useful 
Next, we prepare to investigate the asymptotic behavior of $\eta_0$ as $x\to\infty$.
\begin{lemma}
\label{lem:2.4}
For $x\in\R$, we have 
\begin{align}\label{eq:2.10}
  G_z(b(0)x; 0)
  &\sim -\phi_0(x)^{(p+3)/2},
\\\label{eq:2.11}
  G_\om(b(0)x; 0)
  &\sim -\phi_0(x)^{-p+3}
\end{align}
as $x\to\infty$. 
\end{lemma}
%%%%%%%%%%
%%%%%%%%%%
\begin{proof}
From \eqref{eq:2.7}, we have
\begin{align*}
    G_z(z; 0)
    =-\sqrt{G(z; 0)W(a(0)G(z; 0); 0)}.
\end{align*}
Combined with \eqref{eq:2.5} and $W(s; 0)\sim s^{(p+1)/2}$ as $s\downarrow 0$, this yields \eqref{eq:2.10}. 

From $ W(a(0)s; 0)^{3/2}\sim s^{3(p+1)/4}$ as $s\downarrow0$
and $-\frac{3}{4}p-\frac{1}{4}<-1$, we obtain
\begin{align*}
    F_\omega(\tau; 0)\sim
    -\int_\tau^1s^{-\frac{3}{4}p-\frac{1}{4}}\,ds 
    \sim -\tau^{-3(p-1)/4}\quad\text{as }\tau\downarrow 0.
  \end{align*}
From this and \eqref{eq:2.5}, we obtain
\[  F_\om(G(b(0)x;0); 0)
    \sim -G(b(0)x;0)^{-3(p-1)/4}
    \sim -\phi_0(x)^{-3(p-1)/2}.  \] 
Combined with \eqref{eq:2.8} and \eqref{eq:2.10}, this yields that
\begin{align*}
    G_\om(b(0)x; 0)
    &=-\frac{F_\omega(G(b(0)x; 0) ;0)}{F_\tau(G(b(0)x; 0); 0)}
\\ &=-F_\omega(G(b(0)x; 0) ;0)G_z(b(0)x; 0)
\sim-\phi_0(x)^{-p+3}
\end{align*}
as $x\to\infty$, which completes the proof of \eqref{eq:2.11}.
\end{proof}

We now complete the proof of Theorem~\ref{thm:1.2}.
\begin{proof}[Proof of Theorem~\ref{thm:1.2}]
We first prove the sharp decay/growth estimate \eqref{eq:1.13}. 
It follows from \eqref{eq:2.5}, \eqref{eq:2.6}, \eqref{eq:2.10}, and \eqref{eq:2.11} that
\begin{equation} 
\label{eq:2.12}
    \eta_0(x)
    \underset{x\to\infty}{\sim} a'(0)\phi_0(x)
    -b'(0)x\phi_0(x)^{(p+1)/2}
    -a(0)\phi_0(x)^{-p+2}.
\end{equation}
From the sharp decay of $\phi_0(x)$, the first and second terms of the RHS are estimated as
\begin{align*}
  \phi_0(x)
  \sim x\phi_0(x)^{(p+1)/2}
  \sim x^{-2/(p-1)},
\end{align*}
and the third term is estimated as
\begin{align}
\label{eq:2.13}
   &\phi_0(x)^{-p+2}
    \sim x^{-2/(p-1)+2}.
\end{align}
This means that the third term is the leading term of $\eta_0$ for large $x>0$. Therefore, \eqref{eq:1.13} follows from \eqref{eq:2.12}, $a(0)>0$, and \eqref{eq:2.13}.

Next, we prove \eqref{eq:1.14}. We note that $\partial_\om(\phi_\om')=\eta_\om'$ for $\om>0$ and $x\in \R$ because $(x,\om)\mapsto \phi_\om(x)$ is $C^2$ on $(0,\infty)\times \R$. Therefore, by using the formula \eqref{eq:2.3}, we have
%it follows from \eqref{eq:2.3} that for $\om>0$ and $x>\delta>0$ we have
\begin{align*}
  \eta_\om(x)
  -\eta_\om(\delta)
  &=\int_\delta^x\eta_\om'(y)\,dy
  =\int_\delta^x\partial_\om(\phi_\om'(y))\,dy
\\
 &=\int_\delta^x\frac{2W_s(\phi_\om(y)^2; \om)\phi_\om(y)\eta_\om(y)+\phi_\om(y)^2}{2\sqrt{W(\phi_\om(y)^2; \om)}}\,dy
\end{align*}
for $\om>0$ and $x>\delta>0$.\footnote{The integrand contains the indeterminate form $0/0$ with $y=0$, so we avoid the origin in the integral by taking small $\delta>0$ for $x>0$.}
%where $x>\delta>0$
Passing the limit $\om\downarrow0$, we obtain
\begin{align*}
  \eta_0(x)
  -\eta_0(\delta)
  =\int_\delta^x\frac{2W_s(\phi_0(y)^2; 0)\phi_0(y)\eta_0(y)+\phi_0(y)^2}{2\sqrt{W(\phi_0(y)^2; 0)}}\,dy.
\end{align*}
This implies that $\eta_0\in C^1(\R\setminus\{0\})$ and 
\begin{equation} \label{eq:2.14}
  \eta_0'(x)
  =\frac{1}{2\sqrt{W(\phi_0(x)^2;0)}}\{2W_s(\phi_0(x)^2; 0)\phi_0(x)\eta_0(x)
  +\phi_0(x)^2\}\quad\text{for}~x>0.
\end{equation}
%From the asymptotic behaviors of $W(s; 0)$ and $W_s(s; 0)$ at $s=0$ and 
From the fact $\eta_0(x)\sim{-}\phi_0(x)^{-p+2}$~$(x\to\infty)$, we obtain 
\begin{align}
\label{eq:2.15}
\begin{aligned}
  &W(\phi_0(x)^2; 0)
  \sim \phi_0(x)^{p+1}
  \sim x^{-2}\phi_0(x)^2,
\\
  &W_s(\phi_0(x)^2; 0)\phi_0(x)\eta_0(x)
  \sim -\phi_0(x)^2
  \end{aligned}
\end{align} 
as $x\to\infty$. 
%Thus, we obtain $\eta'_0(x)=O(x\phi_0(x))=O(x^{-2/(p-1)+1})$ as $x\to\infty$. 
Combined with \eqref{eq:2.14}, this yields that
\begin{align*}
\eta'_0(x)=O(x\phi_0(x))=O(x^{-2/(p-1)+1})\quad(x\to\infty).
\end{align*}

Finally we prove that 
\begin{align}
\label{eq:2.16}
\eta_0\in C^\infty(\R)\quad\text{and}\quad L_0\eta_0=-\phi_0\quad\text{in}~\R.
\end{align}
It follows from the argument of the second paragraph in this proof that
\begin{align*}
\eta_0\in C^2(\R\setminus\{0\}),
\quad
%\quad\text{and}\quad
\eta_\omega''(x)\underset{\omega\downarrow0}{\to}\eta_0''(x)
\quad
\text{for}~x\in\R\setminus\{0\}.  
\end{align*}
We recall the fact
\begin{align}
\label{eq:2.17}
\phi_\omega(x)\to\phi_0(x), ~\eta_\omega(x)\to\eta_0(x)\quad\text{for all}~ x\in\R.
\end{align}
Since $\eta_\om$ with $\omega>0$ satisfies
\begin{align}
\label{eq:2.18}
L_\om\eta_\om=-\eta_\om''+\om\eta_\om+p\phi_\om^{p-1}\eta_\om-q\phi_\om^{q-1}\eta_\om=-\phi_\om\quad\text{in}~\R,
\end{align}
by passing the limit $\omega\downarrow0$, we obtain 
\begin{align}
\label{eq:2.19}
L_0\eta_0=-\eta_0''+p\phi_0^{p-1}\eta_0-q\phi_0^{q-1}\eta_0=-\phi_0
\quad\text{in}~\R\setminus\{0\}.
\end{align}

As a key step in the proof, we need to check the differentiability of $\eta_0$ at $x=0$. We note from the formulas \eqref{eq:2.5} and \eqref{eq:2.6} that $\{\phi_\om\}_{0<\om<1}$ and $\{\eta_\om\}_{0<\om<1}$ are bounded in $C([0,1])$.
Combined with \eqref{eq:2.18}, this yields that 
$\{\eta_\om''\}_{0<\om<1}$ is bounded in $C([0,1])$. We set
\begin{align*}
    C_0\ce 
    \sup\{\abs[\eta_\om''(x)]: x\in[0,1],\; \omega\in(0,1)\}.
\end{align*}
We note that $\eta_\om'(0)=0$ for $\om>0$, which follows from the fact $\phi_\om'(0)=0$ for all $\om\ge0$. Using the mean value theorem, we deduce that for $x,\om \in(0,1)$ there exist $\theta_1=\theta_1(x;\om)$ and $\theta_2=\theta_2(x;\omega)\in(0,1)$ such that 
\begin{align*}
\abs[ \frac{\eta_\om(x)-\eta_\om(0)}{x} ]
&=\abs[\eta_\om'(\theta_1x)]
\\
&=\abs[\eta_\om'(\theta_1x)-\eta_\om'(0)]
\\
&=\abs[\eta_\om''(\theta_1\theta_2x)]\abs[\theta_1x]\le C_0|x|.
\end{align*}
Passing the limit $\om\downarrow0$, we obtain
\begin{align*}
\abs[ \frac{\eta_0(x)-\eta_0(0)}{x} ]\le C_0|x|.
\end{align*}
This yields that $\eta_0$ is differentiable at $x=0$ and $\eta_0'(0)=0$.

We now complete the proof of \eqref{eq:2.16}. We deduce from the equation \eqref{eq:2.19} that the limit $\eta_0''(x)$ as $x\to0$ exists. Combined with $\eta_0\in C^1(\R)$, this yields that
\begin{align*}
\eta_0\in C^2(\R)\quad\text{and}\quad L_0\eta_0=-\phi_0\quad\text{in}~\R.
\end{align*}
We conclude from a standard bootstrap argument that $\eta_0\in C^\infty(\R)$. 
%This completes the proof.
\end{proof}
%%%%%%%%%
%%%%%%%%%
\begin{remark}
\label{rem:2.5}
From \eqref{eq:2.15}, the two terms on the RHS of \eqref{eq:2.14} may cancel each other out as $x\to\infty$. Although we do not need the sharp decay/growth estimate of $\eta_0'$ to construct the unstable direction, our estimates of $\eta_0'$ would be sharp except for the case $p=2$.
\end{remark}
%%%%%%%%%
\begin{remark}
\label{rem:2.6}
We avoid the direct calculation of the derivatives of $\eta_0$ in our proof for the smoothness because the derivatives based on the formula \eqref{eq:2.9} always contain the indeterminate form $0/0$ at $x=0$. 
%The claim $\eta_0'(0)=0$ is naturally expected from the fact $\phi_\omega'(0)=0$ for all $\omega\ge0$ but the justification of this claim seems to be nontrivial.
\end{remark}

\section{The link between $M'(0)$ and $\eta_0$}
\label{sec:3}

To bridge our analysis and the results of \cite{H21pre, KLT22}, we need to establish the link between $M'(0)$ and $\eta_0$.
The claim 
\begin{equation}
\label{eq:3.1}
% \text{if}~p<7/3 \implies 
  M'(0)
  =\int_\R\phi_0\eta_0\,dx
\end{equation}
is naturally expected from the formula
\begin{align*}
M'(\om)=\int_{\R} \phi_\omega\eta_\omega\, dx \quad \text{for}~\omega>0
\end{align*}
and \eqref{eq:2.17}, but it is actually a nontrivial problem. 
%even for the one-dimensional case. 
From our construction of $\eta_0$, we do not know the existence of dominated functions for $\{\eta_\omega\}_{0<\omega<1}$ over $\R$. 
%\{\pt_\omega\phi_\omega(x)\}_{0<\omega<1}
%%%
In fact, it seems difficult to provide the dominated function for general $p$. Here for the proof of \eqref{eq:3.1}, we use the nontrivial formula of $M'(0)$ in \cite{IK93}.
\begin{theorem}\label{thm:3.1}
If $p<7/3$, $x\mapsto(\phi_0\eta_0)(x)$ is integrable over $\R$ and \eqref{eq:3.1} holds. If $p\ge7/3$, the both sides of \eqref{eq:3.1} equal ${-}\infty$.
%the equality is interpreted as ${-}\infty$.
%the both sides of \eqref{eq:3.1} equal ${-}\infty$.
\end{theorem}
%%%%%%%%%
%%%%%%%%%
\begin{proof}%[Proof of Theorem~\ref{thm:3.1}]
From \eqref{eq:1.3} and \eqref{eq:1.13} we obtain
\begin{align*}
(\phi_0\eta_0)(x)\sim -|x|^{-4/(p-1)+2} \quad\text{as}~|x|\to\infty.
\end{align*}
This yields that
\begin{align*}
p\ge7/3 \iff \int_{\R}\phi_0\eta_0\, dx=-\infty.
\end{align*}
Hence, the latter claim follows from \eqref{eq:1.6}.

%For the proof of the former claim, we prove
For the rest of the proof, we show that
\begin{equation}\label{eq:3.2}
p<7/3\implies
M'(\om) \to \int_\R\phi_0\eta_0\,dx\quad\text{as}~\omega\downarrow0.
\end{equation}
We now introduce the formula in \cite[Lemma~6]{IK93}
\begin{equation}\label{eq:3.3}
    M'(\om)
    =-\frac{1}{4W_s(a(\om); \om)}\int_0^{a(\om)}\frac{K(a(\om))-K(s)}{J(a(\om))-J(s)}\left(\frac{s}{W(s; \om)}\right)^{1/2}\,ds,
\end{equation}
where
\begin{align*}
  K(s)
  &\ce -\frac{5-p}{p+1}s^{(p-1)/2}+\frac{5-q}{q+1}s^{(q-1)/2},&
\\J(s)
  &\ce-\frac{2}{p+1}s^{(p-1)/2}+\frac{2}{q+1}s^{(q-1)/2}
  =-\frac{W(s; 0)}{s}.
\end{align*}
%It follows from 
We note that
\begin{align*}
0=W(a(\om); \om)=\om a(\om)-a(\om)J(a(\om)),
\end{align*}
and so $W(s; \om)=s(J(a(\om))-J(s))$. Thus, we can rewrite \eqref{eq:3.3} as
\begin{align*}
    \notag
    M'(\om)
   &=-\frac{1}{4W_s(a(\om); \om)}\int_0^{a(\om)}\frac{K(a(\om))-K(s)}{(J(a(\om))-J(s))^{3/2}}\,ds
\\
   &=-\frac{a(\om)}{4W_s(a(\om); \om)}\int_0^{1}\frac{K(a(\om))-K(a(\om)s)}{(J(a(\om))-J(a(\om)s))^{3/2}}\,ds.
\end{align*}
It is easily verified that 
\begin{align*}
    &\abs[ \frac{K(a(\om))-K(a(\om)s)}{(J(a(\om))-J(a(\om)s))^{3/2}}]
    \lesssim \frac{1}{\sqrt{1-s}} \quad\text{for $s$ close to $1$},
\\
   &\frac{1}{(J(a(\om))-J(a(\omega)s))^{3/2}} \le \frac{1}{(-J(a(\omega)s))^{3/2}}
    \sim s^{-3(p-1)/4}
    \quad\text{for small }s>0
\end{align*}
uniformly in $\omega\in(0,1)$.
%$
%s^{-3(p-1)/4}\in L^1(0, \delta)\quad \text{for some }\delta>0 \text{ if }p<7/3.
%$
Therefore, by the dominated convergence theorem we obtain 
\begin{equation}\label{eq:3.4}
  \int_0^{1}\frac{K(a(\om))-K(a(\om)s)}{(J(a(\om))-J(a(\om)s))^{3/2}}\,ds
  \to \int_0^{1}\frac{K(a(0))-K(a(0)s)}{(-J(a(0)s))^{3/2}}\,ds.
\end{equation}
On the other hand, by applying the same calculations as in the proof of \cite[Lemma 6]{IK93}, we obtain 
%noting the integrabity of functions appearing in the process under $p<7/3$, 
%({\color{red} integration by parts, Fubini theorem this works for the case $\omega=0$ with $p<7/3$ because of the integrability})
%we can obtain the following formula for $\om=0$:
\begin{equation}
\label{eq:3.5}
  \int_\R\phi_0\eta_0\,dx
  =-\frac{1}{4W_s(a(0); 0)}\int_0^{1}\frac{K(a(0))-K(a(0)s)}{(-J(a(0)s))^{3/2}}\,ds.
\end{equation}
Indeed, in the proof therein the authors use integration by parts and Fubini's theorem, which still works for the case $\omega=0$ with $p<7/3$ in our setting because of the integrability.
%%%%%%%
Hence, combining \eqref{eq:3.3}, \eqref{eq:3.4}, and \eqref{eq:3.5}, we obtain \eqref{eq:3.2}. This completes the proof.
\end{proof}
%%%%%%%%%%%%%
%%%%%%%%%%%%%
\begin{remark}
\label{rem:3.2}
As can be seen in the proof, the limit of $M'(\om)$ as $\om\to0$ is calculated through the nontrivial formula \eqref{eq:3.3}, which is derived by the calculation heavily depending on the one-dimensional case. In the high-dimensional case, the existence of the limit itself becomes a nontrivial problem. In \cite{LN20} the authors deal with this problem for another type of double power nonlinearities
%, where the lower power term is focusing and the higher power term is defocusing, 
and they use the resolvent of the linearized operator around the stationary solution. However, in our case the problem is more delicate from the viewpoint of linear Schr\"odinger theory because our linearized operator contains the potential of critical decay. 
%We will deal with this problem in a forthcoming paper.
%%%%
%Lewin-Nodari
%Higher dimensional case the existence itself of $\lim_{\omega\downarrow0}M'(\om)$ is a nontrivial problem. 
%%%%
%however our case is more delicate from the viewpoint of linearized operator. We deal with this problem in a forthcoming paper.
\end{remark}

\section{Construction of unstable directions}
\label{sec:4}

In this section we assume that $d=1$ and $1<p<q<5$. 
Based on the function $\eta_0$ constructed in Section \ref{sec:2}, we construct the unstable direction and prove Theorem~\ref{thm:1.3}. 
One can see from our proof that the sharp decay/growth estimate of $\eta_0$ plays an essential role. We remark that our arguments in this section also work for high-dimensional cases. 
%in the same way.

%We recall in this case that $\phi_0\in H^1(\R)$. 

First we introduce the cut-off function. Let $\chi\in C_c^\infty(\R)$ be an even function satisfy $0\le \chi\le 1$ and
\begin{align*}
    \chi(x) 
    =\begin{cases}
    1 & \text{if }|x|\le 1,
\\  0 & \text{if }|x|\ge 2.
    \end{cases}
\end{align*}
For $R>0$ we set $\chi_R(x)\ce \chi(x/R)$. 
%Moreover, we define 
We define the function $\psi_R$ on $\R$ by
\begin{align} 
\label{eq:4.1}
  \psi_R
  &=\phi_0+\beta_R\chi_R\eta_0,
  \quad
  \text{where}~\beta_R\ce -\frac{\|\phi_0\|_{L^2}^2}{(\phi_0, \chi_R\eta_0)_{L^2}}. 
\end{align}
We note that 
\begin{align}\label{eq:4.2}
  &\psi_R\in H_{\rad}^1(\R)
 \quad\text{and}\quad 
  (\psi_R, \phi_0)_{L^2}
  =0. 
\end{align}
The main result in this section is the following, which actually contains more information than Theorem \ref{thm:1.3}.
\begin{theorem}
\label{thm:4.1}
Let $d=1$ and $1<p<q<5$. Assume further that $M'(0)\neq0$. 
%Assume further that $2p+q\neq7$.
Then we have
\begin{alignat}{2}
\label{eq:4.3}
&\text{if}~p<7/3, &\quad&\lim_{R\to\infty}\la L_0\psi_R, \psi_R\ra 
  =\la L_0\phi_0, \phi_0\ra
  +\frac{\|\phi_0\|_{L^2}^4}{M'(0)},
\\
\label{eq:4.4}
&\text{if}~p\ge7/3, &&\lim_{R\to\infty}\la L_0\psi_R, \psi_R\ra 
  =\la L_0\phi_0, \phi_0 \ra.
\end{alignat}
\end{theorem}
%%%%
%\begin{remark}
%Although $\phi_0\notin L^2(\R)$ when $p\ge5$ (see \eqref{eq:1.4}), we note that the term $\tbra[L_0\phi_0, \phi_0]$ is well-defined even for this case.
%\end{remark}
%%%%%%
%First we expand $\la L_0\psi_R, \psi_R\ra$ as

For the proof of Theorem \ref{thm:4.1}, we expand $\la L_0\psi_R, \psi_R\ra$ as
\begin{align}\label{eq:4.5}
  \la L_0\psi_R, \psi_R\ra
  =\la L_0\phi_0, \phi_0\ra
  +2\beta_R\la L_0(\chi_R\eta_0), \phi_0\ra
  +\beta_R^2\la L_0(\chi_R\eta_0),\chi_R\eta_0\ra.
\end{align}
%%%%%%%%%%%%%
%%%%%%%%%%%%%
%We calculate the second and third terms of the RHS in \eqref{eq:4.5}.
From \eqref{eq:1.12} and Leibniz's rule we obtain 
\begin{equation*}
  L_0(\chi_R\eta_0)
  =-\chi_R''\eta_0
  -2\chi_R'\eta_0'
  -\chi_R\phi_0. 
\end{equation*}
Using this relation, we have 
%the second term of the RHS of \eqref{eq:4.5} as
%and the third term as%
\begin{align}
  \la L_0(\chi_R\eta_0), \phi_0\ra
  \label{eq:4.6}
  &=-\la\chi_R''\eta_0, \phi_0\ra
  -2\la\chi_R'\eta_0', \phi_0\ra
  -\la\chi_R\phi_0, \phi_0\ra,
\\
  \la L_0(\chi_R\eta_0), \chi_R\eta_0\ra
  \label{eq:4.7}
  &=-\la\chi_R''\eta_0, \chi_R\eta_0\ra
  -2\la\chi_R'\eta_0', \chi_R\eta_0\ra
  -\la\chi_R\phi_0, \chi_R\eta_0\ra.
\end{align}
%Now we estimate the first and second terms in the RHS of \eqref{eq:4.6} and \eqref{eq:4.7}.
We prepare the following lemma.
\begin{lemma}
The following estimates hold: 
\begin{align}
  \label{eq:4.8}
  &|\la\chi_R''\eta_0, \phi_0\ra|+|\la\chi_R'\eta_0', \phi_0\ra|
  =O(R^{-4/(p-1)+1}),
\\\label{eq:4.9}
  &|\la\chi_R''\eta_0, \chi_R\eta_0\ra|+|\la\chi_R'\eta_0', \chi_R\eta_0\ra|
  =O(R^{-4/(p-1)+3})
\end{align}
as $R\to\infty$. In particular, if $p<7/3$, the left-hand sides of \eqref{eq:4.8} and \eqref{eq:4.9} tend to $0$ as $R\to\infty$.
\end{lemma}

\begin{proof}
We note that the supports of $\chi_R'$ and $\chi_R''$ are contained in the set $\{R\le |x|\le 2R\}$, and 
\begin{align*}
  \chi_R'(x)
  &=\frac{1}{R}\chi'(x/R), 
  \quad
  \chi_R''(x)
  =\frac{1}{R^2}\chi''(x/R).
\end{align*}
In the calculations below, we use the following estimate: 
\begin{align*}
  \int_R^{2R}r^\alpha\,dr 
  &\le 
  \begin{cases}
 \int_R^{2R}(2R)^\alpha\,dr&
  (\text{if}~\alpha\ge 0)
\\[3pt]
\int_R^{2R}R^\alpha\,dr&
%  (q<0)
(\text{if}~\alpha<0)  
  \end{cases}
\\&\lesssim R^{\alpha+1}\quad 
  \text{for }R>0. 
\end{align*}
%%%
From \eqref{eq:1.3}, \eqref{eq:1.13}, and \eqref{eq:1.14}, we have
\begin{align*}
(\eta_0\phi_0)(x)&=O(|x|^{-4/(p-1)+2}),
\quad
(\eta_0'\phi_0)(x)=O(|x|^{-4/(p-1)+1})
\end{align*}
as $|x|\to\infty$, so that 
%we have the estimates
\begin{align*}
   &|\la\chi_R''\eta_0, \phi_0\ra|
    \lesssim \frac{1}{R^2}\int_{R}^{2R}r^{-4/(p-1)+2}\,dr 
    \lesssim R^{-4/(p-1)+1},
\\ &|\la\chi_R'\eta_0', \phi_0\ra|
    \lesssim \frac{1}{R}\int_{R}^{2R}r^{-4/(p-1)+1}\,dr
    \lesssim R^{-4/(p-1)+1}
\end{align*}
for large $R>0$, which yields \eqref{eq:4.8}. 

Similarly, from \eqref{eq:1.13} and \eqref{eq:1.14} we have
\begin{align*}
\eta_0(x)^2=O(|x|^{-4/(p-1)+4}),
\quad
(\eta_0'\eta_0)(x)=O(|x|^{-4/(p-1)+3})
\end{align*}
as $|x|\to\infty$, so that
\begin{align*}
   &|\la\chi_R''\eta_0, \chi_R\eta_0\ra|
    \lesssim \frac{1}{R^2}\int_{R}^{2R}r^{-4/(p-1)+4}\,dr 
    \lesssim R^{-4/(p-1)+3},
\\ &|\la\chi_R'\eta_0', \chi_R\eta_0\ra|
    \lesssim\frac{1}{R}\int_{R}^{2R}r^{-4/(p-1)+3}\,dr
    \lesssim R^{-4/(p-1)+3}
  \end{align*}
for large $R>0$, which yields \eqref{eq:4.9}. 
%This completes the proof of \eqref{eq:4.9}.
\end{proof}
%%%%%%%%%%%%%%%%
%%%%%%%%%%%%%%%%

\begin{proof}[Proof of Theorem \ref{thm:4.1}]
We divide the proof into two cases.
\\[3pt]
{\bf Case 1: $p<7/3$}. It follows from \eqref{eq:4.6}, \eqref{eq:4.8}, and 
$\phi_0\in L^2(\R)$ that
\begin{align*}
\la L_0(\chi_R\eta_0), \phi_0\ra
  \to-\|\phi_0\|_{L^2}^2 \quad (R\to\infty).
\end{align*}
Also, it follows from \eqref{eq:4.7}, \eqref{eq:4.9}, and Theorem \ref{thm:3.1} that 
\begin{align*}
\tbra[L_0(\chi_R\eta_0), \chi_R\eta_0]
  \to -\int_\R\phi_0\eta_0\,dx=-M'(0). 
\end{align*}
From the definition of $\beta_R$ and Theorem \ref{thm:3.1}, we obtain
\[\beta_R
  \to -\frac{\|\phi_0\|_{L^2}^2}{\int_\R\phi_0\eta_0\,dx}
  =-\frac{\|\phi_0\|_{L^2}^2}{M'(0)}. 
  \]
%where we recall that $M'(0)\ne 0$ when $2p+q\ne 7$. 
Combined with \eqref{eq:4.5}, these yield \eqref{eq:4.3}.
\\[3pt]
{\bf Case 2: $p\ge7/3$}. For convenience, we introduce 
\[h_p(R)
  \ce\left\{
  \begin{aligned}
  &(\log R)^{-1} &
  &\text{if }p=\frac{7}{3},
\\&R^{-3+4/(p-1)} &
  &\text{if }p>\frac{7}{3}.
  \end{aligned}\right. 
\]
%which in particular $h_p(R)\to 0$ as $R\to 0$.
We note that $h_p(R)\to 0$ as $R\to 0$.
We first prove
\begin{equation}
\label{eq:4.10}
  |\beta_R|=O(h_p(R))\quad\text{as}~R\to\infty.
\end{equation}
%%%%
Since $(\phi_0\eta_0)(x)\sim -|x|^{-4/(p-1)+2}$ as $|x|\to\infty$, there exist $c_0>0$ and $R_0>0$ such that 
\begin{align*}
\phi_0(x)\eta_0(x)
    \le -c_0|x|^{-4/(p-1)+2} \quad\text{for}~ |x|\ge R_0.
\end{align*}
This yields that
\[  \int_{R_0}^{R}\phi_0(x)\eta_0(x)\,dx
    \lesssim -\int_{R_0}^{R}r^{-4/(p-1)+2}\,dr
    \sim -h_p(R)^{-1}
    \to-\infty
\quad(R\to\infty).
\]
Therefore, we have
\begin{align}
\label{eq:4.11}
    (\phi_0, \chi_R\eta_0)_{L^2}
    \le \int_{0}^{R_0}\phi_0(x)\eta_0(x)\,dx
    +\int_{R_0}^{R}\phi_0(x)\eta_0(x)\,dx
    \sim -h_p(R)^{-1}
\end{align}
for large $R$. Therefore we deduce that
\[|\beta_R|
  \lesssim -\frac{1}{(\phi_0, \chi_R\eta_0)_{L^2}}
  \lesssim h_p(R),
  \]
which yields \eqref{eq:4.10}.
%%%%%%%%%%%%

We now prove \eqref{eq:4.4}. It follows from \eqref{eq:4.6}, \eqref{eq:4.8}, and \eqref{eq:4.10} that
\begin{align*}
  |\beta_R\la L_0(\chi_R\eta_0), \phi_0\ra|
  &\lesssim h_p(R)(R^{-4/(p-1)+1}+\|\phi_0\|_{L^2}^2)
\\&\lesssim R^{-2}+h_p(R)\|\phi_0\|_{L^2}^2
  \to 0 \quad(R\to\infty).
\end{align*}
%as $R\to\infty$. 
Similarly as in \eqref{eq:4.11}, we note that 
$
|\la\chi_R\phi_0, \chi_R\eta_0\ra|\lesssim h_p(R)^{-1}
$
for large $R$. Therefore it follows from \eqref{eq:4.7}, \eqref{eq:4.9}, and \eqref{eq:4.10} that
\begin{align*}
  |\beta_R^2\la L_0(\chi_R\eta_0), \chi_R\eta_0\ra|
  \lesssim h_p(R)^2(R^{-4/(p-1)+3}+h_p(R)^{-1})
  \lesssim h_p(R)
  \to 0\quad(R\to\infty).
\end{align*}
%as $R\to\infty$. 
Combined with \eqref{eq:4.5}, these yield \eqref{eq:4.4}.
\end{proof}
%%%%%%%%%%%%
%%%%%%%%%%%%
We are now in a position to complete the proof of Theorem \ref{thm:1.3}.
\begin{proof}[Proof of Theorem~\ref{thm:1.3}]
%The first term of the RHS in \eqref{eq:4.5} is easily estimated as follows.
We first note that 
\begin{align}
\label{eq:4.12}
\tbra[L_0\phi_0, \phi_0]<0.
\end{align}
Indeed, by using the equation~\eqref{eq:2.1}, this follows from that
\begin{align*}
  \la L_0\phi_0, \phi_0 \ra
  &=\|\phi_0'\|_{L^2}^2
  +p\|\phi_0\|_{L^{p+1}}^{p+1}
  -q\|\phi_0\|_{L^{q+1}}^{q+1}
\\&=-(p-1)\|\phi_0'\|_{L^2}^2
  -(q-p)\|\phi_0\|_{L^{q+1}}^{q+1}. 
\end{align*}
We set $\psi=\psi_R$ defined by \eqref{eq:4.1}. Then by taking large $R>0$ the result follows from \eqref{eq:4.2}, Theorem \ref{thm:4.1}, and \eqref{eq:4.12}. 
\end{proof}

\section{Proof of instability}
\label{sec:A}

In this section, we prove that the existence of $\psi\in H^1_{\rm rad}(\R^d)$ satisfying \eqref{eq:1.16} implies the instability in the sense of Definition \ref{def:1.1}. We mainly follow the argument of \cite{O14}, which is based on \cite{G91, GSS87, SS85}.
Here we consider the standing wave $e^{i\om t}\phi_\om$ with $\om\ge0$ for \eqref{eq:1.1} with $1<p<q<1+4/d$, which is enough to obtain our results, but it is clear that our arguments in this section work for a more general setting. 
%Corollary~\ref{cor:1.4}. 
%First, to reduce the instability problem in the whole space $H^1(\R^d)$ to that in the radial function space $\Hr^1(\R^d)$, we define the following.

First of all, we reduce the instability problem to the radial setting.
\begin{definition}\label{def:A.1}
We assume \eqref{eq:1.4} if $\omega=0$. The standing wave $e^{i\om t}\phi_\om$ of \eqref{eq:1.1} is \emph{radially stable} if for any $\ve>0$ there exists $\delta>0$ such that if $u_0\in \Hr^1(\R^d)$ satisfies $\|u_0-\phi_\om\|_{H^1}<\delta$, then the $H^1$-solution $u(t)$ of \eqref{eq:1.1} with $u(0)=u_0$ satisfies $u(t)\in V_\ve$ for all $t\in\R$, where
\begin{align}
  \label{eq:A.1}V_\ve
  &\ce\{v\in \Hr^1(\R^d): \inf_{\theta\in\R}\|v-e^{i\theta}\phi_\om\|_{H^1}<\ve\}.
\end{align}
Otherwise, the standing wave is \emph{radially unstable}.
\end{definition}
%It is well-known that the stability implies the radial stablity. However, it is difficult to find the proof of the fact, so we give the proof here.

For completeness, we give a proof of the following result.
\begin{lemma} \label{lem:A.2}
If the standing wave $e^{i\om t}\phi_\om$ is stable in the sense of Definition~\ref{def:1.1}, then it is radially stable in the sense of Definition~\ref{def:A.1}.
\end{lemma}
%%%%%%%%
%%%%%%%%
\begin{proof}
For $\ve>0$ we set
\begin{align*}
U_\ve \ce \{v\in H^1(\R^d): \inf_{(\theta,y)\in\R\times\R^d}\|v-e^{i\theta}\phi_\om(\cdot-y)\|_{H^1}<\ve\}.
\end{align*}
First, we show that 
\begin{align}\label{eq:A.2}
%  \text{for any }\ve>0
%  \text{ there exists }\delta>0 
%  \text{ such that }
\forall\ve>0~\exists\delta>0\quad\text{s.t.}\quad
  U_\del\cap\Hr^1(\R^d)\subset V_\ve.
\end{align}
For any $\ve>0$, we can take $\del_1, \del_2>0$ such that 
\begin{align}\label{eq:A.3}
  &|y|<\del_1
  \implies \|\phi_\om(\cdot-y)-\phi_\om\|_{H^1}
  <\frac{\ve}{2},
\\\label{eq:A.4}
  &\|\phi_\om(\cdot-2y)-\phi_\om\|_{H^1}
  <\del_2
  \implies |y|<\del_1,
\end{align}
where \eqref{eq:A.4} follows from \cite[Lemma 6.1]{L09}.
%Indeed \eqref{eq:A.3} follows from the continuity of the translation, and \eqref{eq:A.4} follows from \cite[Lemma 6.1]{L09}.
 If we put $\del\ce\frac12\min\{\del_2, \ve\}$, then we see that $U_\del\cap\Hr^1(\R^d)\subset V_\ve$. Indeed, if $v\in U_\del\cap\Hr^1(\R^d)$, there exists $(\theta_v, y_v)$ such that
\begin{align}
\label{eq:A.5}
\|v-e^{i\theta_v}\phi_\om(\cdot-y_v)\|_{H^1}<\del.
\end{align}
By the radial symmetry of $v$ and $\phi_\om$, we also have
\[\|v-e^{i\theta_v}\phi_\om(\cdot+y_v)\|_{H^1}<\del. \]
Therefore we obtain
\begin{align*}
  \|\phi_\om-\phi_\om(\cdot-2y_v)\|_{H^1}
  &=\|\phi_\om(\cdot+y_v)-\phi_\om(\cdot-y_v)\|_{H^1}
\\&\le \|e^{i\theta_v}\phi_\om(\cdot+y_v)-v\|_{H^1}
  +\|v-e^{i\theta_v}\phi_\om(\cdot-y_v)\|_{H^1}
\\&<2\del
  \le \del_2.
\end{align*}
We obtain by \eqref{eq:A.4} that $|y_v|<\del_1$, and by \eqref{eq:A.3} that
\begin{align}
\label{eq:A.6}
\|\phi_\om(\cdot-y_v)-\phi_\om\|_{H^1}<\frac{\ve}{2}.
\end{align}
Therefore, we deduce from \eqref{eq:A.5} and \eqref{eq:A.6} that
\begin{align*}
  \|v-e^{i\theta_v}\phi_\om\|_{H^1}
  &\le \|v-e^{i\theta_v}\phi_\om(\cdot-y_v)\|_{H^1}
  +\|e^{i\theta_v}\phi_\om(\cdot-y_v)-e^{i\theta_v}\phi_\om\|_{H^1}
\\&< \del+\|\phi_\om(\cdot-y_v)-\phi_\om\|_{H^1}
  <\frac{\ve}{2}+\frac{\ve}{2}
  =\ve,
\end{align*}
which completes the proof of \eqref{eq:A.2}.
%	means $v\in V_\ve$.

Now we complete the proof. Let $\ve>0$. Then \eqref{eq:A.2} implies that $U_\gamma\cap \Hr^1(\R^d)\subset V_\ve$ for some $\gamma>0$. 
%Since $e^{i\om t}\phi_\om$ is stable in $H^1(\R^d)$, and \eqref{eq:1.1} is locally well-posed in $\Hr^1(\R^d)$, 
It follows from the assumption of the stability and the local well-posedness in $\Hr^1(\R^d)$ that there exists $\del>0$ such that if $u_0\in \Hr^1(\R^d)$ satisfies $\|u_0-\phi_\om\|_{H^1}<\del$, then the $H^1$-solution $u(t)$ of \eqref{eq:1.1} with $u(0)=u_0$ satisfies 
\begin{align*}
u(t)\in U_\gamma\cap \Hr^1(\R^d)\subset V_\ve
\quad\text{for all}~t\in\R.
\end{align*}
This completes the proof.
\end{proof}
%%%%%%%%%%%%%%
%%%%%%%%%%%%%%

The aim of this section is to show the following.
\begin{proposition}\label{prop:A.3}
Let $d\ge1$, $1<p<q<1+4/d$, and $\om\ge 0$. Assume further that there exists $\psi\in \Hr^1(\R^d)$ such that 
\begin{align*}
  &(\psi, \phi_\om)_{L^2}
  =0,\quad 
  \la S_\om''(\phi_\om)\psi, \psi\ra
  <0.
\end{align*}
Then the standing wave $e^{i\om t}\phi_\om$ of \eqref{eq:1.1} is radially unstable. 
\end{proposition}
In what follows in this section, we impose the assumption of Proposition~\ref{prop:A.3}.

\subsection{Variational lemmas}

We organize the variational properties of $\phi_\om$. 
%the ground state
We define the Nehari functional by
\[K_\om(v)
  = \pt_\lam S_\om(\lam v)|_{\lam=1}
  =\la S_\om'(v), v\ra .
  \]

\begin{lemma}
%Let $1<p<q<1+4/d$ and $\om\ge0$. Then 
$K_\om(\phi_\om)=0$ and 
\begin{align} \label{eq:A.7}
  S_\om(\phi_\om)
  =\inf\{S_\om(v): v\in \Hr^1(\R^d),\ v\ne0,\ K_\om(v)=0\}.
\end{align} 
\end{lemma}

\begin{proof}
From \cite[Theorem 2.1]{FH21} we have $K_\om(\phi_\om)=0$ and 
\begin{align} \label{eq:A.8}
  S_\om(\phi_\om)
  =\inf\{S_\om(v): v\in X_\om,\ v\ne0,\ K_\om(v)=0\},
\end{align} 
where 
\[X_\om
  \ce \begin{cases}
  \hfil H^1(\R^d) &
\text{if}~\om>0,
\\(\dot{H}^1\cap L^{p+1})(\R^d)
  &\text{if}~\om=0.
  \end{cases}
  \]
Since $\phi_\om$ is radial, \eqref{eq:A.7} with $\om>0$ immediately follows from \eqref{eq:A.8}. We note that $\phi_0$ belongs to $H^1(\R^d)$ when $p<1+4/d$ (see \eqref{eq:1.4}). Therefore, we deduce from the fact $H^1(\R^d)\subset(\dot{H}^1\cap L^{p+1})(\R^d)$ that \eqref{eq:A.7} with $\om=0$ holds. 
\end{proof}

\begin{lemma}\label{lem:A.5}
If $v\in \Hr^1(\R^d)$ satisfies $\la K_\om'(\phi_\om), v\ra=0$, then $\la S_\om''(\phi_\om)v, v\ra\ge0$.
\end{lemma}

\begin{proof}
We define the function with two variables $F(s, \gamma)$ by
\[F(s, \gamma)
  \ce K_\om(\phi_\om+sv+\gamma\phi_\om),\quad 
  s, \gamma\in\R.
  \]
We note that $F(0,0)=K_\om(\phi_\om)=0$ and $\pt_\gamma F(0, 0)=\la K_\om'(\phi_\om), \phi_\om\ra\ne0$. The implicit function theorem implies that there exists a function $s\mapsto \gamma(s)$ defined around $s=0$ such that $\gamma(0)=0$ and $F(s, \gamma(s))=0$ for small $|s|$. By differentiating this relation at $s=0$, we obtain 
\begin{align*}
\gamma'(0)=-\frac{F_s(0,0)}{F_\gamma(0,0)}=-\frac{\la K_\om'(\phi_\om), v\ra}{\la K_\om'(\phi_\om), \phi_\om\ra}=0.
\end{align*}
Moreover, by the variational characterization \eqref{eq:A.7}, we see that $s\mapsto S_\om(\phi_\om+sv+\gamma(s)\phi_\om)$ attains the local minimum at $s=0$. Therefore, using $S_\om'(\phi_\om)=0$ and $\gamma'(0)=0$, we obtain
\begin{align*}
  0&\le \left.\frac{d^2}{ds^2}\right|_{s=0}S_\om(\phi_\om+sv+\gamma(s)\phi_\om)
% \\&=\left.\frac{d}{ds}\right|_{s=0}\la S_\om'(\phi_\om+sv+\gamma(s)\phi_\om), v+\gamma'(s)\phi_\om\ra 
  =\la S_\om''(\phi_\om)v, v\ra.
\end{align*}
This completes the proof.
\end{proof}

\subsection{Modulational lemmas}

\begin{lemma}\label{lem:A.6}
There exists a $C^1$-mapping $\theta\colon V_{\ve_0}\to \R/2\pi\Z$ for some $\ve_0>0$ which satisfies $\theta(\phi_\om)=0$ and the following properties.
\begin{enumerate}
\setlength{\itemsep}{2pt}
\item $(e^{i\theta(v)}v, i\phi_\om )_{L^2}=0$ for all $v\in V_{\ve_0}$.
%$v\in V_{\ve_0}$．

\item $\theta(e^{is}v)=\theta(v)-s$ for all $v\in V_{\ve_0}$ and $s\in\R/2\pi\Z$.
\item $\theta'(v)\in \Hr^1(\R^d)$ for all $v\in V_{\ve_0}$. 
\end{enumerate}
\end{lemma}

\begin{proof}
%We only give an outline of the proof. 
Let $F(v, \theta)\ce (e^{i\theta}v, i\phi_\om)_{L^2}$ for $v\in \Hr^1(\R^d)$ and $\theta\in\R$. Then $F(\phi_\om, 0)=(\phi_\om, i\phi_\om)_{L^2}=0$ and $\pt_\theta F(\phi_\om, 0)=\|\phi_\om\|_{L^2}^2>0$. Thus, the implicit function theorem implies that  there exist $\ve_0>0$ and the $C^1$-mapping $\hat{\theta}\colon B(\phi_\om, \ve_0)\to\R$ such that $F(v, \hat{\theta}(v))=0$ for all $v\in B(\phi_\om, \ve_0)$, where $B(\phi_\om, \ve_0)$ is the open ball in $\Hr^1(\R^d)$ with the center $\phi_\om$ and the radius $\ve_0$. After that, taking $\ve_0$ small enough if necessary, one can extend $\hat{\theta}$ to the mapping $\theta$ defined on $V_{\ve_0}$, and see that $\hat{\theta}$ satisfies (1) and (2) (see \cite{L09} for more details). 

Finally we prove (3). We note that $\pt_\theta F(v, \theta)=(e^{i\theta}v, \phi_\om)_{L^2}$ and $\pt_vF(v, \theta)=ie^{-i\theta}\phi_\om$. By differentiating $F(v, \theta(v))=0$ with respect to $v$, we obtain
\[\theta'(v) 
  =-\frac{\pt_v F(v, \theta(v))}{\pt_\theta F(v, \theta(v))}
  =-\frac{ie^{-i\theta(v)}\phi_\om}{(e^{i\theta(v)}v, \phi_\om)_{L^2}}
  \in \Hr^1(\R^d), 
  \]
which gives the result.
%This completes the outline of the proof.
\end{proof}

We define the function $A\colon V_{\ve_0}\to\R$ by
\[A(v)
  = (e^{i\theta(v)}v, i\psi)_{L^2}. \]
We differentiate $v\mapsto A(v)$ and obtain
\begin{align*}
  A'(v)
  &=(e^{i\theta(v)}v, \psi)_{L^2}\theta'(v)
  +ie^{-i\theta(v)}\psi.
\end{align*}
Thus, Lemma~\ref{lem:A.6} (3) implies $A'(v)\in \Hr^1(\R^d)$. Moreover, it follows from the definition of $A$ that
\begin{equation}\label{eq:A.9}
  \sup_{v\in  V_{\ve_0}}\lvert A(v)\rvert<\infty.
\end{equation}
We now define $a\colon V_{\ve_0}\to\Hr^1(\R^d)$ by 
\begin{equation}\label{eq:A.10}
  a(v)= -iA'(v)
  =-(e^{i\theta(v)}v, \psi)_{L^2}i\theta'(v)
  +e^{-i\theta(v)}\psi 
\end{equation} 
and $P\colon V_{\ve_0}\to\R$ by 
\[P(v) 
  = \langle S_\om'(v), a(v)\rangle.
  \]

\begin{lemma} \label{lem:A.7}
Let $u$ be a $H^1$-solution of \eqref{eq:1.1}. Then for $t$ satisfying $u(t)\in V_{\ve_0}$, we have
\begin{equation*}
  \frac{d}{dt}A(u(t))
  =-P(u(t)).
\end{equation*}
\end{lemma}

\begin{proof}
By Lemma \ref{lem:A.6} (2), we have $A(e^{is}v)=A(v)$ for all $v\in V_{\ve_0}$ and $s\in\R$. By differentiating this relation at $s=0$, we obtain $\langle v, a(v)\rangle=0$. Therefore, if $u$ is solution of \eqref{eq:1.1}, i.e., $i\pt_tu=E'(u)$, then we see that for $t$ satisfying $u(t)\in V_{\ve_0}$, 
\begin{align*}
  \frac{d}{dt}A(u(t))
  &=\langle A'(u(t)), \pt_tu(t)\rangle 
  =-\langle E'(u(t)), a(u(t))\rangle
\\&=-\langle S_\om'(u(t)), a(u(t))\rangle
  =-P(u(t)). 
\end{align*}
This completes the proof.
\end{proof}

\subsection{Proof of instability}

\begin{lemma}\label{lem:A.8}
Let $\lam_0>0$ and $\ve_0>0$ be taken smaller enough. Then 
%If $|\lam|<\lam_0$ and $v\in V_{\ve_0}$, then
\begin{align*}
S_\om(v+\lam a(v))
  \le S_\om(v) +\lam P(v)\quad \text{for $|\lam|<\lam_0$ and $v\in V_{\ve_0}$}.
\end{align*}
\end{lemma}

\begin{proof}
The Taylor expansion of $\lam\mapsto S_\om(v+\lam a(v))$ gives
\begin{equation}\label{eq:A.11}
  S_\om(v+\lam a(v))
  =S_\om(v)
  +\lam P(v)
  +\lam^2\int_0^1(1-s)R(\lam s, v)\,ds,
\end{equation}
where 
\[R(\lam, v)
  \ce \langle S_\om''(v+\lam a(v))a(v), a(v)\rangle. \]
By \eqref{eq:A.10}, $(\phi_\om, \psi)_{L^2}=0$, and $\theta(\phi_\om)=0$, we have $a(\phi_\om)=\psi$. This implies $R(0, \phi_\om)=\langle S_\om''(\phi_\om)\psi, \psi\rangle<0$. Thus, for small $|\lam|$ and for $v$ close to $\phi_\om$, we have $R(\lam, v)<0$. Moreover, $R$ is invariant under $v\mapsto e^{is}v$, i.e., $R(\lam, e^{is}v)=R(\lam, v)$, so we can take $\lam_0>0$ and $\ve_0>0$ smaller enough such that $R(\lam, v)<0$ for $|\lam|<\lam_0$ and $v\in V_{\ve_0}$. Hence the result follows from \eqref{eq:A.11}.
\end{proof}

\begin{lemma}\label{lem:A.9}
Let $\lam_0>0$ and $\ve_0>0$ be taken smaller enough. Then for any $v\in V_{\ve_0}$ there exists $\Lam(v)\in(-\lam_0, \lam_0)$ such that $v+\Lam(v)a(v)\ne 0$ and $K_\om(v+\Lam(v)a(v))=0$.
\end{lemma}

\begin{proof}
From Lemma~\ref{lem:A.5} and $\langle S_\om''(\phi_\om)\psi, \psi\rangle<0$, we obtain $\langle K'(\phi_\om), \psi\rangle \ne0$. We may assume that $\langle K'(\phi_\om), \psi\rangle>0$.

Taking $\lam_0$ and $\ve_0$ smaller if necessary, $v+\Lam a(v)\ne0$ for all $v\in V_{\ve_0}$ and $\Lam\in(-\lam_0, \lam_0)$. Let $F(\Lam, v)\ce K_\om(v+\Lam a(v))$. Then $F(0, \phi_\om)=0$ and $\pt_\Lam F(0, \phi_\om)=\langle K'(\phi_\om), \psi\rangle>0$. Therefore, we can take $\Lam_1, \Lam_2\in(-\lam_0, \lam_0)$ such that $\Lam_1<0<\Lam_2$ and $F(\phi_\om, \Lam_1)<0<F(\phi_\om, \Lam_2)$. By the continuity of $v\mapsto F(v, \Lam_j)~(j=1,2)$, we deduce $F(v, \Lam_1)<0<F(v, \Lam_2)$ for $v$ close to $\phi_\om$. Moreover, $F$ is invariant under $v\mapsto e^{is}v$, so we can take small $\ve_0>0$ such that $F(v, \Lam_1)<0<F(v, \Lam_2)$ for all $v\in V_{\ve_0}$. Therefore, the intermediate value theorem implies that $F(v, \Lam(v))=0$ for some $\Lam(v)\in(\Lam_1, \Lam_2)$, which implies the conclusion. 
\end{proof}

\begin{lemma}\label{lem:A.10}
For any $v\in V_{\ve_0}$, we have
\[S_\om(\phi_\om) 
  \le S_\om(v)
  +\lam_0|P(v)|. \]
\end{lemma}

\begin{proof}
For $v\in V_{\ve_0}$, it follows from Lemma~\ref{lem:A.9}, \eqref{eq:A.7}, and Lemma~\ref{lem:A.8} that
\[S_\om(\phi_\om)
  \le S_\om(v+\Lam(v)a(v))
  \le S_\om(v)
  +\Lam(v)P(v)
  \le S_\om(v)
  +\lam_0\lvert P(v)\rvert. 
  \]
This completes the proof.
\end{proof}

\begin{lemma}\label{lem:A.11}
If $u_0 \in V_{\ve_0}$ and $S_\om(u_0)<S_\om(\phi_\om)$, then the solution $u$ of \eqref{eq:1.1} with $u(0)=u_0$ satisfies $u(t_0)\notin V_{\ve_0}$ for some $t_0\in\R$. 
\end{lemma}

\begin{proof}
Suppose that $u(t)\in V_{\ve_0}$ for all $t\in\R$. By \eqref{eq:A.9}, we have
\begin{equation}\label{eq:A.12}
  \sup_{t\in\R}\lvert A(u(t))\rvert <\infty.
\end{equation}
On the other hand, it follows from $S_\om(u_0)<S_\om(\phi_\om)$, the conservation law of $S_\om$, and Lemma~\ref{lem:A.10} that
\[0<S_\om(\phi_\om)
  -S_\om(u_0)
  =S_\om(\phi_\om)
  -S_\om(u(t))
  \le \lam_0\lvert P(u(t))\rvert
  \]
for all $t\in\R$. Thus, there exists $\del>0$ such that the one of the followings holds: 
%\begin{align*}
%\text{(i)}~P(u_\lam(t))\ge \del~(t\in\R)\quad\text{or}\quad
%\text{(ii)}~P(u_\lam(t))\le -\del~(t\in\R). 
%\end{align*}
(i) $P(u(t))\ge \del$ ($t\in\R$) or (ii) $P(u(t))\le -\del$ ($t\in\R$). 

When (i) holds, by using Lemma~\ref{lem:A.7} we obtain $\frac{d}{dt}A(u(t))=-P(u(t))\le-\del$ for all $t\in\R$. This implies $A(u(t))\to-\infty$ as $t\to\infty$. Similarly, when (ii) holds, we see that $A(u(t))\to+\infty$ as $t\to\infty$. In each case we obtain the contradiction with \eqref{eq:A.12}. This completes the proof.
\end{proof}

\begin{proof}[Proof of Proposition~\ref{prop:A.3}]
Since $S_\om'(\phi_\om)=0$ and $\langle S_\om''(\phi_\om)\psi, \psi\rangle<0$, the Taylor expansion of $\lam\mapsto S_\om(\phi_\om+\lam\psi)$ gives 
\begin{align*}
  S_\om(\phi_\om +\lam\psi)
  =S_\om(\phi_\om)
  +\frac{\lam^2}{2}\langle S_\om''(\phi_\om)\psi, \psi\rangle 
  +o(\lam^2)
  <S_\om(\phi_\om)
\end{align*}
for small $|\lam|>0$. Therefore, Lemma \ref{lem:A.11} implies that the solution $u_\lam$  of \eqref{eq:1.1} with $u_\lam(0)=\phi_\om+\lam\psi$ for small $|\lam|>0$ satisfies $u_\lam(t_0)\notin V_{\ve_0}$ for some $t_0\in\R$. This means that the standing wave $e^{i\om t}\phi_\om$ is radially unstable.
\end{proof}
% Let $u_\lam$ be the solution of \eqref{eq:1.1} with $u_\lam(0)=\phi_\om+\lam\psi$. Suppose that the standing wave $e^{i\om t}\phi_\om$ is radially stable. Then for small $\lam$, we have $u_\lam(t)\in V_{\ve_0}$ for all $t\in\R$. Moreover, by \eqref{eq:A.9},
% \begin{equation}\label{eq:A.12}
%   \sup_{t\in\R}\lvert A(u_\lam(t))\rvert <\infty.
% \end{equation}

% On the other hand, it follows from Lemma~%ref{lem:A.8}%, the conservation law for $S_\om$, and Lemma~\ref{lem:A.10} that
% \[0<S_\om(\phi_\om)
%   -S_\om(u_\lam(0))
%   =S_\om(\phi_\om)
%   -S_\om(u_\lam(t))
%   \le \lam_0\lvert P(u_\lam(t))\rvert
%   \]
% for all $t\in\R$. Thus, there exists $\del>0$ such that the one of the followings holds: 
% %\begin{align*}
% %\text{(i)}~P(u_\lam(t))\ge \del~(t\in\R)\quad\text{or}\quad
% %\text{(ii)}~P(u_\lam(t))\le -\del~(t\in\R). 
% %\end{align*}
% (i) $P(u_\lam(t))\ge \del$ ($t\in\R$) or (ii) $P(u_\lam(t))\le -\del$ ($t\in\R$). 
% When (i) holds, by using Lemma~\ref{lem:A.7} we obtain $\frac{d}{dt}A(u_\lam(t))=-P(u_\lam(t))\le-\del$ for all $t\in\R$. This implies $A(u_\lam(t))\to-\infty$ as $t\to\infty$. Similarly, when (ii) holds, we see that $A(u_\lam(t))\to+\infty$ as $t\to\infty$. In each case we obtain the contradiction with \eqref{eq:A.12}. This completes the proof.
% \end{proof}

%\begin{proof}[Proof of Corollary~\ref{cor:1.4}]
%The assertion follows from Proposition \ref{prop:A.3} and Lemma \ref{lem:A.2}.
%\end{proof}

\appendix
\section{Explicit formulas of $\eta_0$ when $q=2p-1$}\label{sec:B}

In this section we derive the explicit formula of $\eta_0$ in Theorem \ref{thm:1.2} for the case $q=2p-1$. It is known (see \cite{O95dp}) that in this case $\phi_\omega$ is expressed in terms of elementary functions as
\begin{align*}
%  %label{eq:B.1}%
    \phi_{\omega}(x) 
    =\l\{\begin{aligned}
   &\l( \frac{ (p+1)\omega}{ \sqrt{1+\frac{(p+1)^2}{p}\omega } \cosh ((p-1)\sqrt{\omega}x) -1}\r)^{1/(p-1)}&
   &\text{if}~\omega>0,
\\ &\l( \frac{ 2(p+1) }{ \frac{(p+1)^2}{p}+(p-1)^2x^2 } \r)^{1/(p-1)}&
   &\text{if}~\omega=0.
\end{aligned}\r.
\end{align*}
Based on this formula, we prove that $\eta_0$ is also expressed in terms of elementary functions as follows.
\begin{theorem}
\label{thm:B.1}
If $q=2p-1$, then $\eta_0$ is expressed as
\begin{equation}\label{eq:B.1}
  \begin{aligned}
    \eta_0(x)
   &=\frac{1}{(p-1)(p+1)}\left(\frac{ 2(p+1) }{ \frac{(p+1)^2}{p}+(p-1)^2x^2 }\right)^{\frac{1}{p-1}+1}
\\ &\quad\times\left(\frac{(p+1)^4}{8p^2}-\frac{(p-1)^2(p+1)^2}{4p}x^2-\frac{(p-1)^4}{24}x^4\right).
\end{aligned}
\end{equation}
% where
% \begin{align}\label{}
%   \Phi_0(x)
%   &\ce\phi_{0}(x)^{p-1}
%   =\frac{ 2(p+1) }{ \frac{(p+1)^2}{p}+(p-1)^2x^2 },
% \\\xi_0(x)
%   &\ce \frac{(p+1)^4}{8p^2}-\frac{(p-1)^2(p+1)^2}{4p}x^2-\frac{(p-1)^4}{24}x^4.
% \end{align}
In particular, $\eta_0(x)\sim-|x|^{-\frac{2}{p-1}+2}$ as $|x|\to\infty$.
\end{theorem}

\begin{proof}
We set 
\[\Psi_\om(x)
  \ce \phi_\om(x)^{-(p-1)}
  =\frac{ \sqrt{1+\frac{(p+1)^2}{p}\omega } \cosh ((p-1)\sqrt{\omega}x) -1}{(p+1)\omega}
  \]
and fix $x\in\R$. By using the expansions 
\begin{align*}
  &\sqrt{1+y}=
  1+\frac{y}{2}
  -\frac{y^2}{8}
  +O(y^4),
  \quad\cosh y=
  1+\frac{y^2}{2}
  -\frac{y^4}{24}
  +O(y^6)
\end{align*}
as $y\to0$, we obtain
\begin{align*}
  \Psi_\om(x)
  &=\phi_0(x)^{-(p-1)}
\\&\quad+\frac{1}{p+1}\left(-\frac{(p+1)^4}{8p^2}
  +\frac{(p-1)^2(p+1)^2}{4p}x^2
  +\frac{(p-1)^4}{24}x^4\right)\om 
\\&\qquad+O(\om^2)
\end{align*}
as $\om\downarrow 0$. This implies that $\lim_{\om\downarrow0}\Psi_\om(x)=\phi_0(x)^{-(p-1)}$ and that
\[\lim_{\om\downarrow 0}\pt_\om\Psi_\om(x)
  =\frac{1}{p+1}\left(-\frac{(p+1)^4}{8p^2}
  +\frac{(p-1)^2(p+1)^2}{4p}x^2
  +\frac{(p-1)^4}{24}x^4\right). \]
Therefore, by using the relation 
\begin{align*}
\eta_0=\lim_{\om\downarrow0}\pt_\om\phi_\om= -\frac{1}{p-1}\lim_{\om\downarrow0}\Psi_\om^{-\frac{1}{p-1}-1}\pt_\om\Psi_\om,
\end{align*}
we obtain \eqref{eq:B.1}.
\end{proof}
%%%%%%%%%%
\begin{remark}
A relevant calculation can be seen in \cite{H18, H22}, the key difference is that we need here to obtain the information up to the first order with respect to $\om$ while the calculation of the previous results corresponds to obtain the information of the zeroth order.  
\end{remark}

%%%%%%%%%%%%%%
%%%%%%%%%%%%%%

\section*{Acknowledgments}

N.F. was supported by JSPS KAKENHI Grant Number JP20K14349. M.H. was supported by JSPS KAKENHI Grant Numbers JP19J01504 and JP22K20337, and by PRIN project 2020XB3EFL. 

\subsection*{Data availability}
Data sharing not applicable to this article as no datasets were generated or analysed during the current study.

\subsection*{Conflict of interest}
The authors declare that they have no conflict of interest.

% \bibliographystyle{amsplain_abbrev_nobysame_nonumber}
% \bibliography{FH23_g}

\end{document}